\documentclass[11pt, a4paper]{article}
\usepackage{amsmath,amssymb,latexsym,graphics,epsfig}
\usepackage{hyperref}
\usepackage{color}
\usepackage{amsthm}
\usepackage{graphicx,url}
\usepackage{amsopn}

\newtheorem{theo}{Theorem}
\newtheorem{lema}[theo]{Lemma}
\newtheorem{corollari}[theo]{Corollary}

\DeclareMathOperator{\tr}{tr}

\DeclareMathOperator{\spec}{sp}

\def\vec0{\mbox{\boldmath $0$}}

\def\A{\mbox{\boldmath $A$}}

\def\Re{\mathbb R}

\begin{document}
\title{A Spectral Characterization of
Strongly \\ Distance-Regular Graphs with Diameter Four.
}

\author{M.A. Fiol
\\ \\
{\small $^b$Universitat Polit\`ecnica de Catalunya, BarcelonaTech} \\
{\small Dept. de Matem\`atica Aplicada IV, Barcelona, Catalonia}\\
{\small (e-mail: {\tt
fiol@ma4.upc.edu})} \\
 }
\date{}
\maketitle

\begin{abstract}
A graph $G$ with $d+1$ distinct eigenvalues is called strongly distance-regular if $G$ itself is distance-regular, and its distance-$d$ graph $G_d$ is strongly-regular. In this note we provide a spectral characterization of those distance-regular graphs with diameter $d=4$ which are strongly distance-regular. As a byproduct, it is shown that all bipartite strongly distance-regular graphs with such a diameter are antipodal.
\end{abstract}


\noindent{\em AMS subject classifcations:} 05C50, 05E30.

\section{Introduction}
For background on distance-regular graphs and strongly regular graphs, we refer the reader to Brouwer, Cohen, and Neumaier \cite{bcn89}, Brouwer and Haemers \cite{bh12}, and Cameron \cite{c78}.
A {\em strongly distance-regular graph} is a distance-regular graph $G$ (of diameter $d$, say) with the property that its distance-$d$ graph $G_d$ is strongly regular.
Known examples of strongly distance-regular graphs are the strongly regular graphs (since $G_d$ is the complement of $G$), the antipodal distance-regular graphs (where $G_d$ is a disjoint union of complete graphs), and all the distance-regular graphs with $d=3$ and third largest eigenvalue $\lambda_2=-1$.
This last result was reported by Brouwer \cite{b84}, and Brouwer, Cohen, and Neumaier \cite{bcn89} and, in fact, the same conclusion was reached by the author \cite{f00} by only requiring the regularity of $G$. In fact there are some infinite families of this type, such as the generalized hexagons and the Brouwer graphs (see, for instance, \cite{bcn89}).

The above situation suggests the open problem of deciding whether or not the above known families of strongly distance-regular graphs exhaust all the possibilities (see \cite[Conjecture 3.6]{f01} or Cameron \cite{c03}).
Going one step further in this direction, here we prove that a distance-regular graph $G$ with  five distinct eigenvalues $\lambda_0>\lambda_1>\cdots >\lambda_4$ (the case of diameter four) is strongly distance regular if and only  an equality involving them, and the intersection parameters $a_1$ or $b_1$, is satisfied.
Then, as a consequence, it is shown that all bipartite strongly distance-regular graphs with such a diameter are antipodal.

\section{The result}
In proving our result, we use the following scalar product:
\begin{equation}
\label{prod}
\langle p,q\rangle_{G}=\frac{1}{n}\tr (p(\A)q(\A))= \frac{1}{n}\sum_{i=0}^d m_i p(\lambda_i)q(\lambda_i), \qquad p,q\in \Re_{d}[x],
\end{equation}
and  the following lemma (see \cite{f97,f01}):
\begin{lema}
\label{lema}
Let $G$ be a distance-regular graph with spectrum $\spec G=\{\lambda_0,\lambda_1^{m_1},\ldots,\lambda_d^{m_d}\}$, where $\lambda_0>\lambda_1>\cdots>\lambda_d$. Then,
\begin{itemize}
\item[$(a)$]
$G$ is $r$-antipodal if and only if
$$
m_i=\frac{\pi_0}{\pi_i}\quad \mbox{\rm ($i$ even)}, \qquad
m_i=(r-1)\frac{\pi_0}{\pi_i}\quad \mbox{\rm ($i$ odd)}.
$$
\item[$(b)$]
$G$ is strongly distance-regular if and only if, for some positive constants $\alpha,\beta$,
$$
m_i\pi_i=\alpha\quad \mbox{\rm ($i$ odd)}, \qquad m_i\pi_i=\beta\quad \mbox{\rm ($i$ even, $i\neq 0$)}.
$$
\end{itemize}
\end{lema}

Now we are ready to prove our main result:

\begin{theo}
\label{teo(basic)}
Let $G$ be a distance-regular graph with $n$ vertices, diameter $d=4$, and distinct eigenvalues $\lambda_0(=k)>\lambda_1>\cdots>\lambda_4$. Then $G$ is strongly distance-regular if and only if
 \begin{equation}
 \label{condition1}
 (1+\lambda_1)(1+\lambda_3)=(1+\lambda_2)(1+\lambda_4)=-b_1.
 \end{equation}
Moreover, in this case, $G$ is antipodal if and only if either,
\begin{equation}
 \label{condition2}
\lambda_1\lambda_3=-k,\qquad \mbox{or}\qquad \lambda_1+\lambda_3=a_1
\end{equation}
\end{theo}

\begin{proof}
Notice first that the multiplicities $m_0(=1),m_1,\ldots, m_4$, satisfy the following equations:
$$
\sum_{i=0}^4 m_i=n,\quad \sum_{i=0}^4 m_i\lambda_i=0,\quad \sum_{i=0}^4 m_i\lambda_i^2=nk,\quad \sum_{i=0}^4 m_i\lambda_i^3=nka_1,
$$
or, in terms of the scalar product \eqref{prod},
\begin{equation}
\label{closed-paths}
\langle 1,1\rangle_G=1,\quad \langle x,1\rangle_G=0,\quad
\langle x^2,1\rangle_G=k,\quad \langle x^3,1\rangle_G=ka_1.
\end{equation}
From this, and expanding the product below, we have that
\begin{equation}
\label{sp1}
\langle (x-\lambda_0)(x-\lambda_2)(x-\lambda_4), 1\rangle_G = ka_1-(\lambda_0+\lambda_2+\lambda_4)k-\lambda_0\lambda_2\lambda_4.
\end{equation}
Moreover, by using \eqref{prod}, we get that
\begin{equation}
\label{sp2}
\langle (x-\lambda_0)(x-\lambda_2)(x-\lambda_4), 1\rangle_G=0 \quad \iff\quad m_1\pi_1=m_3\pi_3, \\
\end{equation}
Thus, from \eqref{sp1} and \eqref{sp2},
\begin{equation}
\label{sp3}
m_1\pi_1=m_3\pi_3\quad \iff\quad (\lambda_1+1)(\lambda_3+1)=a_1+1-k=-b_1
\end{equation}
since $c_1=1$. Reasoning in the same way with $\langle (x-\lambda_0)(x-\lambda_1)(x-\lambda_3), 1\rangle_G=0$, we get:
$$
m_2\pi_2=m_4\pi_4\quad \iff\quad (\lambda_2+1)(\lambda_4+1)=-b_1.
$$
Then, the characterization in \eqref{condition1} follows from Lemma \ref{lema}$(b)$.
To prove the condition in \eqref{condition2}, observe that, from the above and Lemma
\ref{lema}$(a)$ it suffices to show that $m_4\pi_4=m_0\pi_0$ or, equivalently,
$\langle (x-\lambda_1)(x-\lambda_2)(x-\lambda_3), 1\rangle_G=0$.
Now, this leads to the equality
$$
k(\lambda_1+\lambda_2+\lambda_3)+\lambda_1\lambda_2\lambda_3=ka_1=k(k-b_1-1)
$$
which, together with \eqref{sp3}, gives $k(k+\lambda_1\lambda_3)=\lambda_2(k+\lambda_1\lambda_3)$. But this can only occur when $k+\lambda_1\lambda_3=0$ or, equivalently,  $\lambda_1+\lambda_3=a_1$  (use \eqref{sp3} again). This completes the proof.
\end{proof}

For the case of bipartite graphs, the conditions in \eqref{condition2} clearly hold since $\lambda_3=-\lambda_1$ and $a_1=0$. Thus, every bipartite strongly distance-regular graph is antipodal. Besides, the condition \eqref{condition1} turns to be very simple:

\begin{corollari}
A bipartite distance-regular graph $G$ with diameter $d=4$ is strongly distance-regular if and only if $\lambda_1=\sqrt{k}$.
\end{corollari}
\begin{proof}
Apply \eqref{condition1} with $\lambda_2=0$, $\lambda_3=-\lambda_1$, and $\lambda_4=-k$.
\end{proof}
Then, these graphs have spectrum
$$
\{k^1,\sqrt{k}^{n/2-k},0^{2k-2},-\sqrt{k}^{n/2-k},-k^{1}\}
$$
and, in fact, they constitute a well known infinite family (see Brouwer, Cohen and Neumaier \cite[p. 425]{bcn89}:
With $n=2m^2\mu$ and $k=m\mu$, they are precisely the incidence graphs of symmetric $(m,\mu)$-nets, with intersection array
$$
\{k, k-1, k-\mu, 1; 1, \mu, k-1,k\}.
$$
We finish this note with a question: Looking at the comprehensive table of distance-regular graphs in \cite{bcn89}, it turns out that all the strongly distance-regular graphs (bipartite or not) with diameter four are antipodal. Thus, at first sight, it seems that conditions \eqref{condition1} and \eqref{condition2} could be closely related (although we have not been able to prove it). Is that the case?

\noindent{\large \bf Acknowledgments.} 
This note was written while the author was visiting the Department of Combinatorics and Optimization (C\&O), in the University of Waterloo (Ontario, Canada).
The author sincerely acknowledges to the Department of C\&O the hospitality and facilities received.
Also, special thanks are due to Chris Godsil, with whom the author discussed the possibility of a characterization liked the one presented.

Research supported by the
{\em Ministerio de Ciencia e Innovaci\'on}, Spain, and the {\em European Regional Development Fund} under project MTM2011-28800-C02-01, and the {\em Catalan Research Council} under project 2009SGR1387.



\begin{thebibliography}{99}

%

\bibitem{b84}
A.E. Brouwer, Distance regular graphs of diameter 3 and strongly regular graphs,
{\it Discrete Math.} {\bf 49} (1984) 101--103.

\bibitem{bcn89}
A.E. Brouwer, A.M. Cohen, and A. Neumaier, \emph{Distance-Regular Graphs},
Springer-Verlag, Berlin-New York, 1989.

\bibitem{bh12}
A.E. Brouwer and W.H. Haemers, \emph{Spectra of Graphs},
Springer,
2012; available online at \url{http://homepages.cwi.nl/~aeb/math/ipm/}.

%
\bibitem{c78}
P. Cameron, Stronly distance-regular graphs, in {\it Selected Topics in Graph Theory} (L.J. Beineke and R.J. Wilson, eds.), Acadenic Press, New York, pp. 337--360.

\bibitem{c03}
P. Cameron, Research problems from the 18th British Combinatorial Conference, {\em Discrete Math.} {\bf 266} (2003), no. 1-3, 441--451. The 18th British Combinatorial Conference (Brighton, 2001).

%
%
%
%
%
%

\bibitem{f97}
M.A. Fiol, An eigenvalue characterization of antipodal distance-regular graphs, {\em Electron. J. Combin.} {\bf 4} (1997), \#R30.

\bibitem{f00}
M.A. Fiol,
A quasi-spectral characterization of strongly distance-regular graphs, {\em Electron. J. Combin.} {\bf 7} (2000), \#R51.

\bibitem{f01}
M.A. Fiol,
Some spectral characterization of strongly distance-regular graphs, {\em Combin. Probab. Comput.} {\bf 10} (2001), no. 2, 127--135.

%
%
%
%
%
%
%
%

\end{thebibliography}
\end{document}